\documentclass[12 pt]{amsart}
\usepackage{amsmath,mathrsfs,amsfonts,amssymb,xypic,fullpage,setspace, stmaryrd,hyperref, mathtools }
\usepackage{tikz}
\usepackage{wrapfig}

\usetikzlibrary{arrows,chains,matrix,positioning,scopes}
\makeatletter
\tikzset{join/.code=\tikzset{after node path={%
\ifx\tikzchainprevious\pgfutil@empty\else(\tikzchainprevious)%
edge[every join]#1(\tikzchaincurrent)\fi}}}
\makeatother
\tikzset{>=stealth',every on chain/.append style={join},
         every join/.style={->}}
\tikzstyle{labeled}=[execute at begin node=$\scriptstyle,
   execute at end node=$]

\usepackage[margin=.9in]{geometry}

\usepackage[all]{xy}

\DeclareFontFamily{U}{wncy}{}
\DeclareFontShape{U}{wncy}{m}{n}{<->wncyr10}{}
\DeclareSymbolFont{mcy}{U}{wncy}{m}{n}
\DeclareMathSymbol{\Sh}{\mathord}{mcy}{"58}

\newtheoremstyle{named}{}{}{\itshape}{}{\bfseries}{.}{.5em}{\thmnote{#3's }#1}
\theoremstyle{named}

\theoremstyle{plain}
\newtheorem{theorem}{Theorem}

\newtheorem{corollary}{Corollary}
\newtheorem{lemma}{Lemma}
\newtheorem{proposition}{Proposition}

\theoremstyle{remark}

\newtheorem{remark}{Remark}

\theoremstyle{definition}
\newtheorem{definition}{Definition}

\newcommand{\Q}{\mathbf{Q}}
\newcommand{\F}{\mathbf{F}}

\newcommand{\OO}{\mathcal{O}}
\newcommand{\Z}{\mathbf{Z}}

\newcommand{\GL}{\mathrm{GL}_2}

\newcommand{\PGL}{\mathrm{PGL}_2}
\newcommand{\PSL}{\mathrm{PSL}_2}
\newcommand{\rhof}{\bar{\rho}_{f,\lambda}}
\newcommand{\rhofproj}{\bar{\rho}_{f,\lambda}^{\mathrm{proj}}}

\title{modular forms of arbitrary even weight with no exceptional primes}
\author{Jeffrey Hatley}
\date{\today}

\begin{document}

\begin{abstract}
A result of Dieulefait-Wiese proves the existence of modular eigenforms of weight $2$ for which the image of every associated residual Galois representation is as large as possible. We generalize this result to eigenforms of general even weight $k \geq 2$.
\end{abstract}

\maketitle

\section{Introduction}

The purpose of this note is to provide a modest generalization of a theorem of Dieulefait-Wiese. Before stating the result, we briefly recall some terminology and notation.

Let $f = \sum a_n q^n \in S_k (\Gamma_0(N))$ be a normalized cuspidal modular eigenform (henceforth simply called an ``eigenform'') of weight $k \geq 2$ and level $\Gamma_0(N)$ for some integer $N \geq 1$. Let $G_\Q$ denote the absolute Galois group $\mathrm{Gal}(\bar{\Q}/\Q)$. The Fourier coefficients $\{ a_i \}$ generate a number field $K_f$. Let $\OO_f$ be the ring of integers of $K_f$, let $\lambda$ be a maximal ideal in $\OO_f$ with residue characteristic $\ell$, and write $\F_\lambda$ for the extension of $\F_\ell$ generated by $\{ a_i \mod \lambda \}$, the residues of the Hecke eigenvalues. By work of Deligne, there is a Galois representation
\[
\rho_{f,\lambda} : G_\Q \to \GL(\OO_{f,\lambda})
\]
as well as an associated semisimple residual representation 
\[
\bar{\rho}_{f,\lambda} : G_\Q \to \mathrm{GL}_2 (\F_\lambda).
\]
These representations are unramified outside the primes dividing $N \ell \infty$, and $\rhof$ is absolutely irreducible for almost all primes $\lambda$. Upon composing $\bar{\rho}_{f,\lambda}$ with the natural projection $\GL(\F_\lambda) \to \PGL(\F_\lambda)$, we obtain the projective representation 
\[
\rhofproj : G_\Q \to \PGL(\F_\lambda).
\]

By a result of Ribet \cite[Theorem 3.1]{Rib85}, if $f$ does not have complex multiplication (CM), then the image of $\rhofproj$ is ``as large as possible'' for all but finitely many primes $\lambda$. More precisely, for almost all $\lambda$, the image of $\rhofproj$ is either $\PGL(\F_\lambda)$ or $\PSL(\F_\lambda)$ (see also \cite[Corollary 3.2]{DW}). In Section \ref{history} we briefly discuss the history of such results.

\begin{definition}
A maximal ideal $\lambda$ of $\OO_f$ is called \textit{exceptional} if the image of $\rhofproj$ is not $\PGL(\F_\lambda)$ or $\PSL(\F_\lambda)$. We may also say that $\rhofproj$ is exceptional.
\end{definition}

\begin{remark}
Recall that by Dickson's classification, if $\rhof$ is both irreducible and exceptional, then the image must be either dihedral or isomorphic to $A_4$, $S_4$, or $A_5$.
\end{remark}

Thus Ribet's theorem states that if $f$ does not have CM, then it has only finitely many exceptional primes. The following theorem was proved by Dieulefait-Wiese.

\begin{theorem}\label{DW-main-thm}\cite[Theorem 6.2]{DW}
There exist eigenforms $(f_n)_{n \in \mathbf{N}}$ of weight $2$ such that
\begin{enumerate}
\item for all $n$ the eigenform $f_n$ has no exceptional primes, and
\item for a fixed prime $\ell$, the size of the image of $\bar{\rho}_{f_n,\lambda_n}$ for $\lambda_n \vartriangleleft \OO_{f_n}$ is unbounded for running $n$.
\end{enumerate}
\end{theorem}

\begin{remark} The eigenforms $f_n$ in Theorem $\ref{DW-main-thm}$ have some additional technical properties. First, they do not have CM, which is a necessary condition. Second, they have no nontrivial inner twists; this is important for their application to the Inverse Galois problem in \cite{DW}. While the modular forms which we construct in Theorem \ref{main-thm-intro} also enjoy these properties, we will not mention them for the sake of brevity and ease of exposition.
\end{remark}

In this paper, we modify the arguments of \cite{DW} to obtain a version of Theorem \ref{DW-main-thm} for eigenforms of general even weight $k \geq 2$. The main result of this paper is the following.

\begin{theorem}\label{main-thm-intro}
Let $k \geq 2$ be an even integer. There exist eigenforms $(f_n)_{n \in \mathbf{N}}$ of weight $k$ such that
\begin{enumerate}
\item for all $n$ the eigenform $f_n$ has no exceptional primes, and
\item for a fixed prime $\ell$, the size of the image of $\bar{\rho}_{f_n,\lambda_n}$ for $\lambda_n \vartriangleleft \OO_{f_n}$ is unbounded for running $n$.
\end{enumerate}
\end{theorem}

\begin{remark}\label{rmk-ell}
If $f$ is a weight $2$ eigenform with trivial nebentype whose coefficients are all rational, then by the Eichler-Shimura construction, there is an elliptic curve $E / \Q$ associated to $f$. In \cite{Daniels-SC}, Daniels constructed an explicit infinite family of elliptic curves over $\Q$ whose adelic Galois representations have maximal image; in particular, they have no exceptional primes. In fact, Duke and Jones have shown that, in an appropriate sense, almost all elliptic curves have no exceptional primes \cite{Duke,Jones}. 

Thus, the value of Theorem \ref{main-thm-intro} is in providing modular forms which are guaranteed not to correspond to elliptic curves but which nevertheless have no exceptional primes. 
\end{remark}

\subsection{Historical Context}\label{history}

Given a modular form $f$, one can form an adelic Galois representation
\[
\rho_f : G_\Q \to \prod_\lambda \GL(\OO_{f,\lambda})
\]
where $\lambda$ ranges over all maximal ideals of $\OO_f$. In the special case where $f$ corresponds to an elliptic curve $E / \Q$, this is equivalent to the ``full-torsion'' representation
\[
\rho_E : G_\Q \to \varprojlim_{n} \GL(\Z / n \Z) \simeq \GL(\hat{\Z}).
\]
Serre showed that, assuming $E$ does not have CM, the image of $\rho_E$ is {\em open} in a subgroup of index 2 inside $\GL(\hat{\Z})$ \cite[Proposition 22]{Serre72}; this implies that $E$ has finitely many exceptional primes. As mentioned in Remark \ref{rmk-ell}, more recent results have shown that, generically, an elliptic curve has no exceptional primes \cite{Duke,Jones}.

An analogue of Serre's theorem has recently been proven for modular forms. Loeffler showed that the adelic Galois representation attached to an arbitrary non-CM modular form of weight $k \geq 2$ has open image \cite[Theorem 2.3.1]{Loeffler}. This relies on older results of Ribet and Momose which proved that modular forms have finitely many exceptional primes; see for instance \cite[Theorem 3.1]{Rib85}.

Nevertheless, it can be very hard to explicitly identify the exceptional primes for any given modular form. Recent work of Billerey-Dieulefait gives explicit but complicated bounds on the exceptional primes for a modular form of weight $k \geq 2$ and trivial nebentype \cite{BD}.

\section{Preliminaries}

In this section we collect some definitions and basic results which will be needed in Section \ref{section-main} to prove our main result.

\subsection{Tamely dihedral representations}

The notion of \textit{tamely dihedral representations} was first defined by Dieulefait-Wiese in \cite[Section 4]{DW}; their definition was inspired by the notion of \textit{good-dihedral primes} from \cite{KW1}. We first recall some facts regarding Galois representations arising from modular forms. 

Let $f$ be an eigenform, let $K_f$ be its coefficient field and $\OO_f$ its ring of integers, and let $\lambda \mid \ell$ be a prime of $\OO_f$ dividing a rational prime $\ell$. For any rational prime $p$, let $G_p$ denote a decomposition group corresponding to $p$. For the rest of this section, let $p$ denote a prime different from $\ell$. By Grothendieck's monodromy theorem we may associate to the characteristic zero local representation
\[
\rho_{f,\lambda}\mid_{G_p} : G_p \to \GL(\OO_{f,\lambda})
\]
a $2$-dimensional Weil-Deligne representation $\tau_p =(\tilde{\rho}, \tilde{N})$. Here
\[
\tilde{\rho} : W_{\Q_p} \to \GL(K_{f,\lambda})
\]
is a continuous representation of the Weil group of $\Q_p$ for the discrete topology on $\GL(K_{f,\lambda})$, $\tilde{N}$ is a nilpotent matrix in M$_2(K_{f,\lambda})$, and we have the relation
\[
\tilde{\rho} \tilde{N} \tilde{\rho}^{-1} = \mid \cdot \mid^{-1} \tilde{N}
\]
where $\mid \cdot \mid$ is a particular norm map. The standard reference for these things is \cite{Tate}, but another very readable reference is \cite{Gee}.

\begin{definition}\cite[Definition 4.1]{DW}
Let $\Q_{p^2}$ be the unique unramified degree $2$ extension of $\Q_p$. Denote by $W_p$ and $W_{p^2}$ the Weil groups of $\Q_p$ and $\Q_{p^2}$, respectively.

A $2$-dimensional Weil-Deligne representation $\tau_p =(\tilde{\rho}, \tilde{N})$ of $\Q_p$ with values in $K_f$ is called {\em tamely dihedral of order $n$} if $\tilde{N}=0$ and there is a tame character $\psi : W_{p^2} \to K_{f,\lambda}^\times$ whose restriction to the inertia group $I_p$ (which is naturally a subgroup of $W_{p^2}$) is of niveau $2$ (i.e. it factors over $\mathbf{F}_{p^2}^\times$ and not over $\F_p^\times$) and of order $n > 2$, such that $\tilde{\rho} \simeq \mathrm{Ind}_{W_{p^2}}^{W_p} \psi$.

We say that an eigenform $f$ is {\em tamely dihedral of order $n$} at the prime $p$ if the Weil-Deligne representation $\tau_p$ at $p$ associated to $f$ is tamely dihedral of order $n$. 
\end{definition}

\begin{remark}
In terms of the local Langlands correspondence, $f$ can only be tamely dihedral at $p$ if it is {\em supercuspidal} at $p$. Recent work of Loeffler-Weinstein \cite{LW} has made it possible to test modular forms for the property of being tamely dihedral using the LocalComponent package of \cite{sage}. Thus, in theory one can find explicit examples of the modular forms whose existence is guaranteed by Theorem \ref{main-thm-intro}; however, as the proof of the theorem will indicate, these modular forms are expected to have very large level, and their construction seems beyond the scope of current computing capabilities. 
\end{remark}

\begin{proposition}\label{levelraise}
Let $f \in S_k(N,\chi_{triv})$ be a newform of odd level $N$ and trivial nebentype such that for all $\ell \mid N$
\begin{enumerate}
\item $\ell \mid \mid N$ or
\item $\ell^2 \mid \mid N$ and $f$ is tamely dihedral at $\ell$ of order $n_\ell > 2$ or
\item $\ell^2 \mid N$ and $\rho_{f,t}(G_\ell)$ can be conjugated to lie in the upper triangular matrices such that the elements on the diagonal all have odd order for some prime $t \nmid \ell$.
\end{enumerate}
Let $\{ p_1 , \ldots, p_r \}$ be any finite set of primes.

Then for almost all primes $p \equiv 1 \mod 4$ there is a set $S$ of primes of positive density which are completely split in $\Q(i, \sqrt{p_1}, \ldots, \sqrt{p_r})$ such that for all $q \in S$ there is a newform $g \in S_k (Nq^2, \chi_{triv})$ which is tamely dihedral at $q$ of order $p$ and for all $\ell \mid N$ we have
\begin{enumerate}
\item $\ell^2 \mid \mid N$ and $g$ is tamely dihedral at $\ell$ of order $n_\ell > 2$ or
\item $\rho_{g,t}(G_\ell)$ can be conjugated to lie in the upper triangular matrices such that the elements on the diagonal all have odd order for some prime  $t \nmid \ell$.
\end{enumerate}
\end{proposition}

\begin{proof}
This is \cite[Proposition 5.4]{DW}.
\end{proof}

\subsection{Local $\ell$-adic representations}

Let $f=\sum a_n q^n$ be an eigenform, and let $\lambda$ be a prime of $\OO_f$ lying above the rational prime $\ell$. Recall that $f$ is said to be \textit{ordinary at} $\lambda$ if $a_\ell \not\equiv 0$ (mod $\lambda$); otherwise $f$ is said to be \textit{nonordinary at} $\lambda$. Let $G_\ell$ be a decomposition group at $\ell$ and $I_\ell$ its inertia group. 

The following theorem is due to Deligne, Fontaine, and Edixhoven.

\begin{theorem}\label{localform}
Assume $f$ is weight $k$ and $\rhof$ is irreducible.
\begin{enumerate}
\item If $k \geq 2$ and $f$ is ordinary at $\lambda$ then
\[
\rhof\mid_{I_\ell} \ \simeq \left( \begin{matrix} \chi_\ell^{k-1} & \ast \\ 0 & 1   \end{matrix} \right)
\]
where $\chi_\ell$ is the (reduction of the) $\ell$-adic cyclotomic character.
\item If $2 \leq k \leq \ell + 1$ and $f$ is nonordinary at $\lambda$ then
\[
\rhof\mid_{I_\ell} \ \simeq \left( \begin{matrix} \phi^{k-1} & 0 \\ 0 & \phi^{\ell(k-1)}   \end{matrix} \right)
\] 
where $\phi$ is a fundamental character of niveau $2$.
\end{enumerate}
\end{theorem}
\begin{proof}
We refer to reader to \cite[Theorem 1.2]{BG} and the remark which follows it.
\end{proof}

Thus, the image of inertia under $\rhof$ can be identified with the image of $\chi_\ell^{k-1}$ or $\phi^{(l-1)(k-1)}$ depending on whether $f$ is ordinary or nonordinary at $\lambda$. In particular, we have the following corollary.

\begin{corollary}\label{cyclicinertia}
Assume $\ell > k$. Let $\mathcal{I} = \rhofproj(I_\ell)$.
\begin{enumerate}
\item If $f$ is ordinary at $\lambda$, then $\mathcal{I}$ is cyclic of order $n = (\ell-1)/\mathrm{gcd}(\ell-1,k-1) \geq 2$. If $\ell > 5k-4$, then $n > 5$.

\item If $f$ is nonordinary at $\lambda$, then $\mathcal{I}$ is cyclic of order $n = (\ell + 1)/\mathrm{gcd}(\ell + 1,k-1) \geq 2$. If $\ell > 5k-4$, then $n > 5$.
\end{enumerate}
\end{corollary}
\begin{proof}
This follows immediately from Theorem \ref{localform}; see also \cite[Lemma 1.2]{BD}
\end{proof}

We conclude this section with a lemma which is the crucial ingredient for generalizing from weight $2$ forms to weight $k$ forms. The first argument of this sort, for the $k=2$ case, goes back to Ribet (see the proof of \cite[Proposition 2.2]{Rib97}). For higher weights, see \cite[Section 3.3]{BD}, which we follow closely.

Let $\mathcal{G}=\rhofproj(G_\Q)$ be the projective image of $\rhof$, and suppose $\mathcal{G}$ is dihedral. Then $\mathcal{G}$ fits into an exact sequence of the form
\[
0 \to \mathcal{Z} \to \mathcal{G} \to \{ \pm 1\} \to 0 
\]
where $\mathcal{Z}$ is cyclic. This corresponds to a tower of fields 
\[
\Q \subset E \subset L
\]
with Galois groups 
\[
\mathrm{Gal}(L / \Q) \simeq \mathcal{G}, \ \mathrm{Gal}(E / \Q) \simeq \{ \pm 1\}, \ \mathrm{Gal}(L/E) \simeq \mathcal{Z}.
\]
We thus obtain a quadratic character $\epsilon : G_\Q \to \{ \pm 1 \}$ whose kernel cuts out $E$. 
\begin{lemma}\label{localinertia}
If $\ell > 5k-4$, then $\epsilon$ is unramified at $\ell$.
\end{lemma}

\begin{proof}
By Corollary \ref{cyclicinertia}, $\mathcal{I}$ is cyclic of order $> 5$. Since $\mathcal{I} \subset \mathcal{G}$, we must have $\mathcal{I} \subset \mathcal{Z}$. Thus $I_\ell$ is contained in the kernel of $\epsilon$. 
\end{proof}

\section{Main result}\label{section-main}

In order to prove our main theorem, we must first prove a version of \cite[Proposition 6.1]{DW} for eigenforms of general weight $k \geq 2$, after which the proof of our theorem will follow easily.

\begin{proposition}\label{mainprop}
Let $p,q,t,u$ be distinct odd primes and let $N$ be an integer which is divisible by every odd prime $p \leq 5k-4$. Let $p_1,\ldots, p_m$ be the prime divisors of $2N$. Let $f \in S_k(Nq^2u^2, \chi)$ be an eigenform without CM which is tamely dihedral of order $p^r > 5$ at $q$ and tamely dihedral of order $t^s > 5$ at $u$. Assume that $q$ and $u$ are completely split in $\Q(i,\sqrt{p_1},\ldots,\sqrt{p_m})$ and that $( \frac{q}{u} )= ( \frac{u}{q} )=1$.

Then $f$ does not have any exceptional primes, i.e. for all maximal ideals $\lambda$ of $\OO_f$, the image of $\rhofproj$ is $\PSL(\F_\lambda)$ or $\PGL(\F_\lambda)$. 
\end{proposition}

\begin{proof}
The proof is similar to the proof of \cite[Proposition 6.1]{DW}, which we follow closely. Let $\lambda$ be any maximal ideal of $\OO_f$ and suppose it lies over the rational prime $\ell$. By our ``tamely dihedral'' hypotheses, $\rhof$ is irreducible, since if $\ell \notin \{p,q\}$, then already $\rhof \mid_{G_q}$ is irreducible, and if $\ell \in \{p, q\}$, then $\ell \notin \{t,u \}$, hence $\rhof \mid_{G_u}$ is irreducible. 

Now suppose the image of $\rhofproj$ is a dihedral group. This means that $\rhofproj$ is the induction of a character of a quadratic extension $E / \Q$, i.e.
\[
\rhofproj \simeq \mathrm{Ind}_E^\Q(\alpha)
\]
for some character $\alpha$ of $\mathrm{Gal}(\bar{\Q} / E)$. By the ramification properties of $\rhofproj$, we know 
\begin{equation}\label{Eeq}
E \subset \Q(i, \sqrt{\ell}, \sqrt{q}, \sqrt{u}, \sqrt{p_1}, \ldots, \sqrt{p_m}).
\end{equation}

First assume that $\ell \notin \{p,q\}$. In this case, we have
\[
\rhofproj \mid_{D_q} \simeq \mathrm{Ind}_{\Q_{q^2}}^{\Q_q} (\psi) \simeq \mathrm{Ind}_{E_\mathfrak{q}}^{\Q_q} (\alpha)
\]
where $\mathfrak{q}$ is a prime in $\OO_E$ lying over $q$ and where $\psi$ is a niveau $2$ character of order $p^r$. This implies that $q$ is inert in $E$, but by assumption $q$ is totally split in $\Q(i,\sqrt{u},\sqrt{p_1},\ldots,\sqrt{p_m})$, so from (\ref{Eeq}) we deduce that 
\[
\ell \notin \left\{ u, p_1, \ldots, p_m \right\}.
\]
 In particular, we see that $\ell \nmid 2Nu$, so by our choice of $N$, we conclude that $\ell > 5k-4$. Thus by Lemma \ref{localinertia} our quadratic field $E$ cannot ramify at $\ell$, so we can refine (\ref{Eeq}) to
 \[
 E \subset \Q(i, \sqrt{q}, \sqrt{u}, \sqrt{p_1}, \ldots, \sqrt{p_m}),
\]
with $E$ totally split in the latter. But now the fact that $q$ is inert in $E$ implies that $E=\Q$ rather than a quadratic extension, and this contradiction implies that $\ell \in \{p , q \}$ and in particular $\ell \notin \{ t, u \}$. Upon exchanging the roles $q \leftrightarrow u$, $p \leftrightarrow t$, and $r \leftrightarrow s$, running this argument again leads to a contradiction, hence the image of $\rhofproj$ is not dihedral.

If $\lambda$ is exceptional and the image of $\rhofproj$ is not dihedral, then by Dickson's classification, the only other possibilities for the image are $A_4$, $S_4$, and $A_5$. But the image of $\rhofproj$ contains an element of order $>5$ by Corollary \ref{cyclicinertia}, so none of these are possible. 
\end{proof}

We may now prove our main theorem. The proof is essentially the same as the proof of \cite[Theorem 6.2]{DW}.


\begin{proof}[Proof of Theorem \ref{main-thm-intro}]
Start with some newform $f_1 \in S_k(\Gamma_0(N))$ for $N$ of squarefree level. Note that modular forms of level $\Gamma_0(N)$ never have CM when $N$ is squarefree. Let $p_1, \ldots, p_m$ be the prime divisors of $6N$. 

Let $B_1 > 0$ be any bound. Take $p$ to be any prime bigger than $B$ provided by Proposition \ref{levelraise} applied to $f$ and the set $\{ p_1, \ldots, p_m \}$. We thus obtain an eigenform $g \in S_k(\Gamma_0(Nq^2))$ which is tamely dihedral at $q$ of order $p$ for some prime $q$. Now apply Proposition \ref{levelraise} to the form $g$ and the set $\{ q, p_1, \ldots, p_m \}$ to obtain a prime $t > B$ different from $p$ and an eigenform $h \in S_k(\Gamma_0(Nq^2 u^2))$ which is tamely dihedral at $u$ of order $t$ for some prime $u$. By Proposition \ref{mainprop}, $h$ does not have any exceptional primes. 

Thus we take $f_2 = h$ and take a new bound $B_2 > B_1$. Inductively we obtain a family $(f_n)_{n \in \mathbf{N}}$ and the image of inertia grows without bound in this family.
\end{proof}

\section*{Acknowledgments} It is a pleasure to thank Keenan Kidwell for his careful reading of an earlier version of this paper. We are very grateful to the anonymous referees for offering helpful comments and corrections in an extremely timely fashion.

\end{document}